\documentclass[11pt]{amsart}
\usepackage{amsmath}
\usepackage{amsfonts}
\usepackage{amssymb}
\usepackage{graphicx}

\providecommand{\U}[1]{\protect\rule{.1in}{.1in}}
\providecommand{\U}[1]{\protect\rule{.1in}{.1in}}
\providecommand{\U}[1]{\protect\rule{.1in}{.1in}}
\providecommand{\U}[1]{\protect\rule{.1in}{.1in}}
\newtheorem{theorem}{Theorem}[section]
\theoremstyle{plain}

\newtheorem{corollary}{Corollary}[section]

\newtheorem{lemma}{Lemma}[section]

\newtheorem{remark}{Remark}[section]

\numberwithin{equation}{section}
\input{tcilatex}

\begin{document}

\title[On $\kappa$-noncollapsed complete noncompact shrinking GRS which split at infinity]
{On $\kappa$-noncollapsed complete noncompact shrinking gradient Ricci solitons  which
split at infinity}

\author{Bennett Chow}
\address{Department of Mathematics \\ 
University of California \\ 
San Diego, La Jolla, CA 92093}
\email{benchow@math.ucsd.edu}
\author{Peng Lu}
\address{Department of Mathematics \\
 University of Oregon \\
  Eugene, OR 97403}
  \email{penglu@uoregon.edu}
\thanks{P.\thinspace L. is partially supported by a grant from the Simons Foundation.}

\date{\today}

\begin{abstract}
We discuss some geometric conditions under which a complete noncompact
shrinking gradient Ricci soliton will split at infinity.


\end{abstract}

\maketitle


\section{Introduction}

For complete noncompact shrinking gradient Ricci solitons (GRS), global
splitting theorems have been proven by Lichnerowicz (\cite{Lichnerowicz}),
Fang, Li and Zhang (\cite{FangEtc}) and Munteanu and Wang (\cite{MunteanuWang1},
\cite{MunteanuWang3}). In the more general context of
Bakry--Emery manifolds, these works prove splitting theorems under the
assumptions of the existence of geodesic lines and conditions on the
potential functions $f$. As a special case, one recovers the splitting
theorem of Cheeger and Gromoll (\cite{CheegerGromollRicSplit}) for complete
noncompact Riemannian manifolds with nonnegative Ricci curvature.

For complete noncompact shrinking GRS one may conjecture that their
curvatures must be bounded, their Ricci curvatures cannot be everywhere
positive, and either they split as the Riemannian product of $\mathbb{R}$
with a one-lower-dimensional shrinking GRS or they have quadratic curvature
decay. This conjecture is known to be true in dimensions at most $3$, where
quadratic curvature decay implies being a Gaussian shrinker, but the
conjecture remains open for dimensions at least $4$. Evidence toward this
conjecture are in the works of Chen (\cite{ChenStrongUniqueness}) and Cao and Zhou (\cite{CaoZhou}).

\medskip
In this paper, for $\kappa $-noncollapsed complete noncompact shrinking GRS
with bounded curvature, we discuss a splitting theorem at infinity (Theorem \ref{thm main de BC} below), a.k.a.,
\emph{dimension reduction} (see  Hamilton \cite[\S 22]{Hamilton1995}).

\medskip
In the rest of this section we collect some elementary facts regarding shrinking GRS
which will be used later, the reader may see \cite[\S 4.1]{ChowLuNi} for more details.
Let $(\mathcal{
M}^{n},g,f,\lambda )$ be a complete noncompact shrinking GRS with $f$
normalized, so that
\begin{equation}
\func{Rc}_{g}+\nabla _{g}^{2}f=\frac{\lambda }{2}g\quad \text{and}\quad
R_{g}+|\nabla _{g}f|_{g}^{2}=\lambda f,  \label{Split-Infinity-00}
\end{equation}
where $\lambda \in \mathbb{R}_{+}$.
Here and below, $\func{Rc}$, $\nabla ^{2}f$, and $R$ denote the Ricci tensor, Hessian of $f$,
and scalar curvature associated to a Riemannian metric, respectively.
In a later application, we shall rescale
the $\lambda =1$ case. If $(\mathcal{M}^{n},\bar{g},f)$ satisfies $\func{Rc}
_{\bar{g}}+\nabla _{\bar{g}}^{2}f=\frac{1}{2}\bar{g}$ and $R_{\bar{g}
}+|\nabla _{\bar{g}}f|_{\bar{g}}^{2}=f$, then $(\mathcal{M},\lambda ^{-1}
\bar{g},f)$ satisfies $\func{Rc}_{\lambda ^{-1}\bar{g}}+\nabla _{\lambda
^{-1}\bar{g}}^{2}f=\frac{\lambda }{2} \cdot \lambda ^{-1}\bar{g}$ and $R_{\lambda
^{-1}\bar{g}}+|\nabla _{\lambda ^{-1}\bar{g}}f|_{\lambda ^{-1}\bar{g}
}^{2}=\lambda f$.

For a shrinking GRS  $(\mathcal{M}^{n}, g, f, \lambda )$
we define diffeomorphisms $\varphi _{t}:\mathcal{M}\rightarrow \mathcal{M}$ by $
\frac{\partial }{\partial t}\varphi _{t}\left( x\right) =\frac{1}{1-\lambda t
}$ $\left( \nabla _{g}f\right) \left( \varphi _{t}\left( x\right) \right) $ and
$\varphi _{0}=\func{id}_{\mathcal{M}}$, $t\in (
-\infty ,\lambda ^{-1})$. Let $f(t)\doteqdot f\circ \varphi _{t}$ and $
g(t)\doteqdot (1-\lambda t)\varphi _{t}^{\ast }g$. Then $g(t)$ is a solution
to the Ricci flow, which satisfies
\begin{align}
& \frac{\partial f}{\partial t}(x,t)=\frac{1}{1-\lambda t}
|\nabla _{g}f|_{g}^{2}\left( \varphi _{t}\left( x\right) \right) =|\nabla
_{g(t)}f(t)|_{g(t)}^{2}(x),  \label{eq shrinking GRS RF 1}  \\
& \limfunc{Rc}{}_{g(t)}+\nabla _{g(t)}^{2}f(t)=\frac{\lambda }{2(1-\lambda t)}
g(t),  \label{eq shrinking GRS RF 2} \\
& R_{g(t)}+\left\vert \nabla _{g(t)}f(t)\right\vert
_{g(t)}^{2}=\dfrac{\lambda }{1-\lambda t}f(t).  \label{Split-Infinity-01}
\end{align}
By a result of Cao and Zhou (\cite{CaoZhou} and \cite
{HaslhoferMuller}), for $(\mathcal{M}, g, f, \lambda )$ we have the estimate
\begin{equation}f\left( x\right) \geq \frac{1}{4}
\left( (\lambda ^{1/2}d_{g}\left( x,O\right) -5n)_{+}\right) ^{2}. \label{eq f lower bdd}
\end{equation}

\section{Splitting at infinity for limits of shrinking GRS}

We say that a sequence $\{(\mathcal{M}_{i}^{n},g_{i},f_{i}, \lambda
_{i},x_{i})\}_{i=1}^\infty$ of pointed complete noncompact shrinking GRS is
\textbf{admissible} if it satisfies the following conditions.

\begin{enumerate}
\item $\func{Rc}_{g_i}+\nabla_{g_i}^{2}f_i =\frac{\lambda_i}{2}g$, where $
0<\lambda_i \leq \lambda$ for some $\lambda <\infty$.

\item The $f_i$ are normalized by $R_{g_i}+|\nabla_{g_i} f_i|_{g_i}^{2}=
\lambda_i f_i$.

\item $\lambda _{i}d_{g_{i}}(x_{i},O_{i}) \rightarrow \infty $, where $O_{i}$ is a minimum
point of $f_{i}$.

\item $\func{inj}_{g_{i}}(x_{i})\geq \iota $, for some  $\iota >0$.

\item For any given radius $\rho >0$, there exists a constant $C_{1}(\rho
)<\infty $ such that
\begin{equation}
|\func{Rm}_{g_{i}}|_{g_{i}}(x)\leq C_1(\rho)\quad \text{for }x\in B_{\rho
}^{g_{i}}(x_{i})\text{ and for all }i.  \label{localcurv}
\end{equation}

\item  For any given $\rho >0$, there is a constant $C_2(\rho)< \infty$ such that Ricci tensor satisfies
$|\nabla_{g_i} \func{Rc}_{g_i}|_{g_i} \leq C_2 (\rho)$ on the ball $B^{g_i}_{\rho}(x_i)$ for each $i$.

\end{enumerate}

We say that a Riemannian manifold \textbf{splits} if it isometric to the
product of a line and a Riemannian manifold. The reason we assume
that the $\lambda _{i}$ are bounded in the definition of admissibility above
is that otherwise we would allow for rescalings of asymptotically conical shrinking GRS
(although these are not counterexamples due to conditions 5 and 6), whose corresponding
limits do not split. Such a splitting result (Theorem \ref{thm compactness shrinker seq}
below) is the main result of this section.

We will need the following version of the compactness theorem
for a sequence of Riemannian manifolds.
We assume the reader is familiar with the notion of
$C^{k,\alpha}$ pointed Cheeger--Gromov convergence.

\begin{theorem}\label{thm compact nabla rc}
Let $\{ (\mathcal{M}^n_I,g_I,x_I)\}_{I=1}^\infty$ be a sequence of  pointed smooth complete Riemannian
manifolds of dimension $n$. Suppose that\emph{:}\smallskip

\noindent \emph{(a)} the injectivity radius
$\func{inj}_{g_I}(x_I) \geq \iota $ for all $I$, where $\iota$ is a positive constant\emph{;}\smallskip

\noindent \emph{(b) (bounded curvature at bounded distance)}
given any $\rho >0$, there is a constant $C_1(\rho)$ such that the Riemann curvature tensors
satisfy $|\func{Rm}_{g_I}|_{g_I} \leq C_1 (\rho)$ in the ball $B^{g_I}_{\rho}(x_I)$ for each $I$\emph{;}
and\smallskip

\noindent \emph{(c)}  given any $\rho >0$, there is a constant $C_2(\rho)$
such that the Ricci tensors satisfy
$|\nabla_{g_I} \func{Rc}_{g_I}|_{g_I} \leq C_2 (\rho)$
in the ball $B^{g_I}_{\rho}(x_I)$ for each $I$.\smallskip

Then, for any $\alpha \in (0,1)$, the sequence $\{ (\mathcal{M}^n_I,g_I,x_I ) \}$
subconverges in the $C^{2,\alpha}$ pointed Cheeger--Gromov sense to a pointed $C^{2,\alpha}$  complete
Riemannian manifold $(\mathcal{M}^n_\infty, g_\infty, x_\infty)$.
\end{theorem}

\begin{proof}
By assumption (b) and a theorem of  Cheeger, Gromov, and Taylor (\cite{CGT82}), we have the following.
Given any $\rho >0$ there is a constant
$\iota_0 = \iota_0(\rho,\iota,n)  >0$ such that the injectivity radius $\func{inj}_{g_I}(x)\geq \iota_0$
 for any $x \in B_{\rho}^{g_I}(x_I)$ and for any $I$. Fix an $\alpha \in (0,1)$.
Then we can use (b) and a theorem of Jost and Karcher (\cite{JoKa82})
to further conclude the following.
There is a constant $r_0 = r_0(\rho,\iota,n) \in (0,\iota_0(\rho,\iota,n))$ such that,
for each $x \in B_{\rho}^{g_I}(x_I)$, in harmonic coordinates on
 $B_{r_0}^{g_I}(x)$ the metric tensor coefficients
  $(g_{I})_{ij}$ satisfy
 the following  estimates:

 \medskip
 \noindent (b1) The $(g_I)_{ij}$ on  $B_{3r_0/4}^{g_I}(x)$ have
 uniformly (independent of $I$) bounded $C^{1,\alpha}$ norms in the local harmonic coordinates.\smallskip

 \noindent (b2) For each $y \in B_{3r_0/4}^{g_I}(x)$ we have $\frac{1}{10} (\delta_{ij})
 \leq ((g_I)_{ij} (y) ) \leq 10 (\delta_{ij})$.

 \medskip
 For any  $x \in B_{\rho}^{g_I}(x_I)$ it follows from (b1),  (b2),  and assumption (c) that:

\medskip
\noindent (c1) The Ricci tensor coefficients $(\func{Rc}_{g_I})_{ij}$, in the harmonic coordinates on
$B_{3r_0/4}^{g_I}(x)$, have
uniformly (independent of $I$)  bounded $C^\alpha$ norm.

\medskip
 \noindent  Since Ricci tensor coefficients in harmonic  coordinates are given by
 \[
 (\func{Rc}_{g_I})_{ij} =-\frac{1}{2} (g_I)^{kl} \frac{\partial^2 (g_I)_{ij}}{\partial x^k
  \partial x^l} + \cdots ,
 \]
where the dots indicate lower order terms involving at most one derivative of the
 metric  (\cite[Lemma 4.1 ]{DeKa81}), by the Schauder estimates for elliptic \textsc{pde},
 (b1), (b2), and (c1), we have that:

 \medskip
  \noindent (c2) The metric tensor  coefficients $(g_I)_{ij}$, in harmonic coordinates, have uniform
 (independent of $I$) $C^{2,\alpha}$  estimates  on  $B_{r_0/2}^{g_I}(x)$ for any
 $x \in B_{\rho}^{g_I}(x_I)$.

\medskip
Note that Greene and Wu, and separately Peters (see Greene's survey \cite{Greene93}), proved
the following theorem.  If a sequence of Riemannian manifolds $\{ (\mathcal{M}^n_I,g_I )
\}_{I=1}^\infty$  satisfies the following conditions:

\medskip
\noindent (a1)  the injectivity radius $\func{inj}_{g_I} \geq \iota >0 $ for all $I$,\smallskip

\noindent  (b3) for some constant $C_1$ the Riemann curvature tensor $| \func{Rm}_{g_I} |_{g_I}
\leq C_1$ on $\mathcal{M}_I$  for each $I$, and\smallskip

\noindent (d) (uniformly bounded diameter) for some constant $C_3$ the diameter
 $\func{diam}(\mathcal{M}_I,$ $g_I ) \leq C_3$
for each $I$,

\medskip
\noindent then  the sequence  $\{ (\mathcal{M}^n_I,g_I )
\}$ subconverges in the $C^{1,\alpha}$ Cheeger--Gromov
sense to a $C^{1,\alpha}$ Riemannian
manifold $(\mathcal{M}^n_\infty,g_\infty )$ of dimension $n$.\smallskip

If the Riemannian manifolds in Theorem \ref{thm compact nabla rc} have uniformly
bounded diameter, then we can use (c2) and the above compactness result to conclude
Theorem \ref{thm compact nabla rc}.

If the Riemannian manifolds in Theorem \ref{thm compact nabla rc} do not have uniformly
bounded diameter, then  we can adjust slightly Hamilton's
proof of a compactness result for pointed Riemannian
manifolds under the assumption that all derivatives of curvature tensor are bounded
 (\cite{Hamilton1995B}) to address the noncompact limit. Regarding this,
 a key modification is to replace the normal
 coordinates used in constructing the limit manifolds by
 the harmonic coordinates constructed above.
 We omit the details here.
 ~~~ $\square$
\end{proof}

Now we can prove a splitting theorem at infinity for the limit of a sequence of shrinking GRS.

\begin{theorem}\label{thm compactness shrinker seq}
Let $\{(\mathcal{M}_{i}^{n},g_{i},f_{i},\lambda _{i},x_{i})\}$ be an admissible sequence of pointed complete
noncompact shrinking GRS.
Let $(\mathcal{M}^n_\infty,g_\infty, x_\infty)$ be the pointed $C^{2,\alpha}$  Cheeger--Gromov limit of some
subsequence of $\{(\mathcal{M}_{i}^{n},g_{i},x_{i})\}$ given by Theorem \ref{thm compact nabla rc}.
Then $(\mathcal{M}_{\infty}^n,
g_{\infty })$ is isometric to the product of $\mathbb{R}$ with a complete $C^{2,\alpha}$
 Riemannian manifold $(\mathcal{N}^{n-1},h)$.
\end{theorem}

\begin{remark}
Note that $h$ is possibly flat.  In particular, this is true for noncompact shrinking GRS $\{(\mathcal{M}^{n},
g,f, 1,x_{i})\}$  with $x_i \rightarrow \infty$
 under condition that $|\func{Rc}|\rightarrow 0$, in which case the shrinking GRS
is asymptotically conical.
\end{remark}

\begin{proof}
 Define the function
\begin{equation*}
F^{(i)}\left( x\right) \doteqdot  2 \lambda _{i}^{-1/2} (\sqrt{f_{i}}(x)  -\sqrt{
f_{i}}(x_{i}))\quad \text{for }x\in \mathcal{M}_{i}.
\end{equation*}
Using the diffeomorphisms in the definition of Cheeger--Gromov convergence,
we can transplant $F^{(i)}$ to a sequence of functions which are defined on balls
$B_{\rho}^{g_\infty}(x_\infty) \subset \mathcal{M}_\infty$ for arbitrary $\rho>0$,
as long as we choose $i$ large enough.
The basic idea of the proof is that the limit of a subsequence of these transplanted  functions
provides the splitting of $(\mathcal{M}_\infty,g_\infty)$.
A version of this idea was introduced in  \cite[p. 383]{ChowLuNi}
by the first named author of this article.

By (\ref{Split-Infinity-00}), the gradient of $F^{(i)}$ satisfies
\begin{equation*}
|dF^{(i)}|_{g_{i}}=\lambda _{i}^{-1/2}  \cdot \frac{|df_{i}|_{g_{i}}}{f_{i}^{1/2}}
=\left( 1-\frac{R_{g_{i}}}{\lambda _{i}f_{i}}\right) ^{1/2}\leq 1\quad \text{
on }\mathcal{M}_{i},
\end{equation*}
since $R_{g_{i}}\geq 0$ by Chen \cite{ChenStrongUniqueness}.
Fix any $\rho >0$ and let $y \in B^{g_i}_\rho (x_i)$. Using (\ref{eq f lower bdd}), we have
\begin{align}
f_i(y) &\geq \frac{1}{4} \left ((\lambda_i^{1/2} d_{g_i}(y,O_i) -5n)_+ \right )^2  \label{eq f i lower bdd ball} \\
& \geq \frac{1}{4} \left (\lambda_i^{1/2} (d_{g_i}(x_i,O_i) -\rho) -5n \right )^2 \notag \\
& \geq \frac{1}{4} \left (\lambda_i^{1/2} d_{g_i}(x_i,O_i) - \lambda^{1/2} \rho -5n \right )^2, \notag
\end{align}
for $i$ large enough such that $\lambda_i^{1/2} d_{g_i}(x_i,O_i) - \lambda^{1/2} \rho -5n >0$.
Here we have used $\lambda_i \leq \lambda$.
Hence, we have
\[
\frac{R_{g_{i}}(y)}{\lambda _{i}f_{i}(y)} \leq \frac{C_1(\rho)}{\frac{1}{4} \left (\lambda_i
d_{g_i}(x_i,O_i) - \lambda \rho -5n \lambda^{1/2} \right )^2}.
\]
 By the admissible assumption that $\lambda_i
d_{g_i}(x_i,O_i) \rightarrow \infty$, we get that for any fixed $\rho >0$
\[
\lim_{i \rightarrow \infty} \sup_{y \in B^{g_i}_\rho (x_i)}
\frac{R_{g_{i}}(y)}{\lambda _{i}f_{i}(y)}  =0.
\]
Hence we get
\begin{equation}
\lim_{i \rightarrow \infty} \sup_{y \in B^{g_i}_\rho (x_i)}  |dF^{(i)}|_{g_{i}}(y) =1.
\label{eq grad F i est}
\end{equation}

\medskip
Next we consider the Hessian of $F^{(i)}$. In local coordinates and using
(\ref{Split-Infinity-00}), we compute that
\begin{align*}
\left\vert \nabla _{g_{i}}^{2}F^{(i)}\right\vert _{g_{i}}& =\lambda
_{i}^{-1/2}\left\vert \frac{\nabla _{g_{i}}^{2}f_{i}}{f_{i}^{1/2}}-\frac{1}{2
}\frac{df_{i}\otimes df_{i}}{f_{i}^{3/2}}\right\vert _{g_{i}} \\
& =\lambda _{i}^{-1/2}\left( \frac{\left\vert -\func{Rc}_{g_{i}}+\frac{
\lambda _{i}}{2}g_{i}\right\vert _{g_{i}}}{f_{i}^{1/2}}+\frac{1}{2}\frac{
|df_{i}|_{g_{i}}^{2}}{f_{i}^{3/2}}\right) \\
& \leq \frac{ \left\vert \func{Rc}_{g_{i}}\right\vert
_{g_{i}} +\frac{\sqrt{n}+1}{2}\lambda _{i}}{(\lambda_{i} f_{i})^{1/2}}.
\end{align*}
Fix any $\rho >0$ and let $y \in B^{g_i}_\rho (x_i)$. Using (\ref{eq f i lower bdd ball}), we have
\begin{equation}
\left\vert \nabla _{g_{i}}^{2}F^{(i)}\right\vert _{g_{i}}(y)
\leq \frac{ C_1(\rho)+ \frac{\sqrt{n}+1}{2} \lambda} { \frac{1}{2} \left (\lambda_i
d_{g_i}(x_i,O_i) - \lambda \rho -5n \lambda^{1/2} \right )}. \label{eq Hess F i est bdd}
\end{equation}
Hence, we get
\begin{equation}
\lim_{i \rightarrow \infty} \sup_{y \in B^{g_i}_\rho (x_i)}  |\nabla _{g_{i}}^{2} F^{(i)}|_{g_{i}}(y) =0.
\label{eq Hess F i est}
\end{equation}

\medskip
Thirdly, we consider the third covariant derivatives of $F^{(i)}$. Again, using local coordinates, we
compute that
\begin{align*}
& \left\vert \nabla _{g_{i}}^{3}F^{(i)}\right\vert _{g_{i}} \\
& =\lambda _{i}^{-1/2}\left\vert \frac{\nabla _{\ell jk}^{3}f_{i}}{
f_{i}^{1/2}}-\frac{\nabla _{jk}^{2}f_{i}\nabla _{\ell }f_{i}+\nabla _{\ell
k}^{2}f_{i}\nabla _{j}f_{i}+\nabla _{\ell j}^{2}f_{i}\nabla _{k}f_{i}}{
2f_{i}^{3/2}}+\frac{3}{4}\frac{\nabla _{\ell }f_{i}\nabla _{j}f_{i}\nabla
_{k}f_{i}}{f_{i}^{5/2}}\right\vert _{g_{i}} \\
& \leq \lambda _{i}^{-1/2}\left( \frac{\left\vert \nabla _{g_{i}}\func{Rc}
_{g_{i}}\right\vert _{g_{i}}}{f_{i}^{1/2}}+\frac{3 \left (|\func{Rc}
_{g_{i}}|_{g_{i}}+\frac{\lambda _{i}\sqrt{n}}{2} \right )\left\vert \nabla
_{g_{i}}f_{i}\right\vert _{g_{i}}}{2f_{i}^{3/2}}+\frac{3}{4}\frac{\left\vert
\nabla _{g_{i}}f_{i}\right\vert _{g_{i}}^{3}}{f_{i}^{5/2}}\right) \\
& \leq \frac{\left\vert \nabla _{g_{i}}\func{Rc}
_{g_{i}}\right\vert _{g_{i}}}{(\lambda_i f_{i})^{1/2}}+\frac{3 \lambda_i^{1/2} \left( 2|\func{Rc}
_{g_{i}}|_{g_{i}}+\lambda _{i}(\sqrt{n}+1)\right) }{4f_{i}},
\end{align*}
where we used $\func{Rc}_{g_{i}}+\nabla _{g_{i}}^{2}f_{i}=\frac{\lambda _{i}
}{2}g_{i}$ and $|\nabla _{g_{i}}f_{i}|_{g_{i}}^{2}\leq \lambda _{i}f_{i}$.
Fix any $\rho >0$ and let $y \in B^{g_i}_\rho (x_i)$. Using (\ref{eq f i lower bdd ball})
and condition 6 in the admissible assumption, we have
\begin{align}
  \left\vert \nabla _{g_{i}}^{3}F^{(i)}\right\vert _{g_{i}}(y) \leq & \frac{C_2(\rho)}{ \frac{1}{2}
 \left (\lambda_id_{g_i}(x_i,O_i) - \lambda \rho -5n \lambda^{1/2} \right )} \label{eq third deriv F i est bdd} \\
 & + \frac{6C_1(\rho) \lambda^{3/2} + 3( \sqrt{n}+1) \lambda^{5/2}}{\left (\lambda_id_{g_i}(x_i,O_i) -
  \lambda \rho -5n \lambda^{1/2} \right )^2}.  \notag
\end{align}
Hence, we get
\begin{equation}
\lim_{i \rightarrow \infty} \sup_{y \in B^{g_i}_\rho (x_i)}  |\nabla _{g_{i}}^{3} F^{(i)}|_{g_{i}}(y) =0.
\label{eq third deriv F i est}
\end{equation}

\medskip
By the admissible assumption and Theorem \ref{thm compact nabla rc}, we know that
$\{(\mathcal{M}_i,g_i,x_i) \}$ subconverges to $(\mathcal{M}_\infty,g_\infty,x_\infty)$
in the $C^{2,\alpha}$ Cheeger--Gromov sense for any $\alpha  \in (0,1)$.
There exists an exhaustion $\{U_i \}$ of $\mathcal{M}_\infty$ by relatively compact open sets
and embeddings $\psi_i : U_i \rightarrow \mathcal{M}_i$ such that $\overline{U}_i \subset U_{i+1}$,
$\psi_i (x_\infty) = x_i$,
 and $ \psi_i^* g_i \rightarrow g_\infty$ in the $C^{2,\alpha}$ topology on any compact subset of
 $\mathcal{M}_\infty$.

Let $\tilde{g}_i \doteqdot \psi_i^* g_i$ be metrics  on
$\mathcal{M}_\infty$ and let $\tilde{F}^{(i)} \doteqdot F^{(i)} \circ \psi_i$ be the transplanted functions  on
$\mathcal{M}_\infty$.
 Fix any $\rho >0$. By the convergence of $\tilde{g}_i$  to $g_\infty$
and by equations (\ref{eq grad F i est}), (\ref{eq Hess F i est bdd}), and (\ref{eq third deriv F i est bdd}),
there is a constant $C_3>1$ such that the following bounds hold
for any $y \in B^{g_\infty}_\rho (x_\infty)$:
\begin{align}
 \tilde{F}_i(x_\infty)= & \, 0, \label{eq F tilde bdd} \\
 |d \tilde{F}^{(i)}|_{g_{\infty}} (y) \leq & \, C_3, \label{eq grad F tilde bdd} \\
  \left\vert \nabla _{\tilde{g}_{i}}^{2} \tilde{F}^{(i)}\right\vert _{g_{\infty}} (y) \leq &
\, \frac{ C_3(C_1(\rho)+ \frac{\sqrt{n}+1}{2} \lambda ) } { \frac{1}{2} \left (\lambda_i
d_{g_i}(x_i,O_i) - \lambda \rho -5n \lambda^{1/2} \right )},  \label{eq Hess F tilde bdd} \\
 \left\vert \nabla _{\tilde{g}_{i}}^{3} \tilde{F}^{(i)}\right\vert _{g_{\infty}}(y)
\leq \, & \frac{C_3C_2(\rho)}{ \frac{1}{2}
 \left (\lambda_id_{g_i}(x_i,O_i) - \lambda \rho -5n \lambda^{1/2} \right )}  \label{eq grad power 3 F tilde bdd} \\
& + \frac{C_3(6C_1(\rho) \lambda^{3/2} +3( \sqrt{n}+1) \lambda^{5/2})}{\left (\lambda_id_{g_i}(x_i,O_i) -
  \lambda \rho -5n \lambda^{1/2} \right )^2}.  \notag
\end{align}
From (\ref{eq F tilde bdd}) and (\ref{eq grad F tilde bdd}), we know that $\tilde{F}^{(i)}$ and
$d \tilde{F}^{(i)}$
are uniformly bounded on $( B^{g_\infty}_\rho (x_\infty), g_\infty)$.

In local coordinates, we  compute that the components of Hessian are
\begin{equation}
 (\nabla _{\tilde{g}_{i}}^{2} \tilde{F}^{(i)})_{kl} = \frac{\partial^2 \tilde{F}^{(i)} }{\partial x^k \partial x^l}
-\Gamma_{kl}^p(\tilde{g}_{i})  \frac{\partial \tilde{F}^{(i)} }{\partial x^p}. \label{eq F i tilde x  x der}
\end{equation}
By the $C^{2,\alpha}$ convergence of $\tilde{g}_i$ to $g_\infty$, we know that the $\Gamma_{kl}^p(\tilde{g}_{i})$
are uniformly bounded on $( B^{g_\infty}_\rho (x_\infty), g_\infty)$, and hence the
$ \frac{\partial^2 \tilde{F}^{(i)} }{\partial x^k \partial x^l}$ are uniformly bounded on $( B^{g_\infty}_\rho
(x_\infty), g_\infty)$, so the $\nabla _{\tilde{g}_{\infty}}^{2} \tilde{F}^{(i)}$ are uniformly bounded on
$( B^{g_\infty}_\rho (x_\infty), g_\infty)$.

Equation (\ref{eq grad power 3 F tilde bdd}) implies that components $(\nabla _{\tilde{g}_{i}}^{2}
\tilde{F}^{(i)})_{kl}$
have uniformly bounded $C^\alpha$ norm with respect to the metric $g_\infty$.
By (\ref{eq F i tilde x  x der}) and that the $\Gamma_{kl}^p(\tilde{g}_{i})$ have uniformly bounded $C^\alpha$ norm,
we conclude that the $ \frac{\partial^2 \tilde{F}^{(i)} }{\partial x^k \partial x^l}$ are uniformly bounded
in $C^\alpha$ norm on $( B^{g_\infty}_\rho (x_\infty), g_\infty)$, so the $\nabla _{\tilde{g}_{\infty}}^{2}
\tilde{F}^{(i)}$  are uniformly bounded in the $C^\alpha$ norm on $( B^{g_\infty}_\rho (x_\infty), g_\infty)$.

\vskip .2cm
Hence we can conclude that $\tilde{F}^{(i)}$  subconverges to a function $F_\infty$ in the $C^{2,\alpha}$ norm
on $(\mathcal{M}_\infty, g_\infty)$.
It follows from (\ref{eq grad F i est}), (\ref{eq Hess F i est}), and (\ref{eq third deriv F i est}) that
the limit function $F_{\infty }$ on $\mathcal{M}
_{\infty }$ satisfies $F_{\infty }(x_{\infty })=0$, $\left\vert \nabla F_{\infty
}\right\vert _{g_{\infty }}\equiv 1$, and $\left\vert \nabla ^{2}F_{\infty
}\right\vert _{g_{\infty }}\equiv 0$.  This yields a Riemannian splitting for the
limit. Note that if the metric $g_\infty$ is smooth, then
$F_{\infty }$ is also smooth.
~~~ $\square$
\end{proof}

The following is a parabolic version of Theorem \ref{thm compactness shrinker seq} under a stronger assumption.

\begin{corollary} \label{cor type I sol split}
Let $\{(\mathcal{M}_{i}^{n},g_{i},f_{i},\lambda _{i},x_{i})\}$ be a sequence of pointed complete
noncompact shrinking GRS with $f_i$ normalized so that
$\func{Rc}_{g_i}+\nabla_{g_i}^{2}f_i =\frac{\lambda_i}{2}g$, where $
0<\lambda_i \leq \lambda$ for some $\lambda <\infty$ and where
  $R_{g_i}+|\nabla_{g_i} f_i|_{g_i}^{2}=
\lambda_i f_i$.
We assume that the $(\mathcal{M}_{i}^{n},g_{i}), \, i =1, 2, \ldots$, have
uniformly bounded curvatures
$| \func{Rm}_{g_i}|_{g_i} \leq C < \infty$ and are uniformly $\kappa$-noncollapsed \emph{(}below a fixed scale\emph{)}
for some $\kappa >0$.
Let $(\mathcal{M}^n_i,g_i(t))$ be the Ricci flow solution in canonical form associated to the shrinking
GRS $(\mathcal{M}_{i}^{n},g_{i},f_{i},\lambda _{i})$
as defined preceding \emph{(\ref{eq shrinking GRS RF 1})}.

Then the sequence $\{ (\mathcal{M}_i^n, g_i(t), x_i) \}$ subconverges in the $C^\infty$ pointed Cheeger--Gromov
sense to an ancient solution of the Ricci flow   $ (\mathcal{M}_\infty^n, g_\infty(t), x_\infty),
\, t \in (-\infty, \lambda^{-1})$. Moreover, $ (\mathcal{M}_\infty^n, g_\infty(t))$
is isometric to the product of $\mathbb{R}$ with a complete
 $\kappa$-noncollapsed \emph{(}below some fixed scale\emph{)}
 Type I \emph{(}at $t= -\infty$\emph{)} ancient solution $(\mathcal{N}^{n-1},h(t))$.
Note that $h(t)$ is possibly flat.
\end{corollary}

\begin{proof}

Since $|\func{Rm}_{g_i}|_{g_i} \leq C$, from the discussion preceding
(\ref{eq shrinking GRS RF 1}), we have
\[
|\func{Rm}_{g_i(t)}|_{g_i(t)} \leq
\frac{C}{1-\lambda_i t}  \leq \frac{C}{1-\lambda t}.
\]
By Shi's derivative estimates, on any compact subinterval $I$ of $(-\infty, \lambda^{-1})$
there exist constants $C_{k,n,I}$ (independent of $i$) such
that $|\nabla_{g_i(t)} ^{k}\func{Rm}_{g_i(t)}|_{g_i (t)} \leq C_{n,k, I}$ on $\mathcal{M}_i$ for $k\in \mathbb{
N}$, $t \in I$,  and  all $i$.

It follows from the uniformly $\kappa$-noncollapsed assumption and the uniform curvature bound
assumption that there is a constant $\iota >0$ such that $\func{inj}_{g_i(0)}(x_{i}) \geq \iota$ for all $i$.
 By Hamilton's Cheeger--Gromov compactness theorem for
solutions of Ricci flow and by passing to a subsequence if necessary, $\{(\mathcal{M}_i^n
,g_i(t), x_{i})\}$, $t\in (-\infty ,\lambda^{-1})$, converges to an ancient solution $(
\mathcal{M}_{\infty }^{n},g_{\infty }(t),x_{\infty })$, $t\in (-\infty , \lambda^{-1})$,
of the Ricci flow satisfying $|\func{Rm}_{g_{\infty }(t)}|_{g_{\infty }(t)} \leq \frac{C}{1- \lambda t}$.
This implies that $g_{\infty }(t)$ is Type I at $t =- \infty$

Fix a $t \in (-\infty, \lambda^{-1})$. By (\ref{eq shrinking GRS RF 2}), $g_i(t)$ is a shrinking GRS.
It is easy to check that $\{ (\mathcal{M}_i^n
,g_i(t), f_i(t), $ $\frac{\lambda_i}{1- \lambda_i t}, x_{i}) \} $ is an admissible sequence of pointed complete
 noncompact
shrinking GRS. In particular, condition 3 in the definition of admissibility follows from
\[
C^{-1} d_{g_i(t)}(x_i,O_i) \leq d_{g_i}(x_i,O_i) \leq C d_{g_i(t)}(x_i,O_i)
\]
for some constant $C$ independent of $i$.
This inequality is due to the fact that a uniform curvature bound
implies uniform metric equivalence for the Ricci flow (see, for example, \cite[Lemma 6.10]{ChowLuNi}).
The uniform curvature bound also implies uniform volume equivalence
of unit balls. Hence, condition 4 that $\func{inj}_{g_i(t)}(x_i) \geq \iota_1 >0$ for an admissible
sequence follows from a theorem of Cheeger, Gromov,
and Taylor (\cite{CGT82}).
Now we can apply Theorem \ref{thm compactness shrinker seq} to the sequence  $\{ (\mathcal{M}_i^n
,g_i(t), f_i(t), \frac{\lambda_i}{1- \lambda_i t}, x_{i}) \} $ to conclude that the limit $(
\mathcal{M}_{\infty }^{n},g_{\infty }(t))$ is  isometric to the product of $\mathbb{R}$ with a complete
Riemannian manifold $(\mathcal{N}^{n-1},h(t))$ for each $t$.
The constancy of the splitting of $\mathcal{M}_{\infty }^{n}$
into $\mathbb{R} \times \mathcal{N}^{n-1}$ for different $t$ follows from the smooth dependence on $t$
of the functions
\[
F^{(i)} (x,t) \doteqdot 2 \left (\frac{\lambda_i}{1-\lambda_i t} \right )^{-1/2} \cdot
 ( \sqrt{f_i}(x,t) -  \sqrt{f_i}(x_i,t) ),
\]
whose limits are used to define the splitting.  Now the theorem is proved.
~~~ $\square$
\end{proof}

\begin{remark}
If we combine Deng and Zhu's discussion (\cite[p. 12]{DZ15}) and
Corollary \ref{cor type I sol split} and apply them to a complete noncompact shrinking K\"{a}hler GRS,
then we obtain a splitting of the form $(\mathbb{R} \times S^1) \times W^{n-1}$ or $\mathbb{C} \times W^{n-1}$,
where $W$ is of complex dimension $n-1$.
\end{remark}

\section{Splitting at infinity of shrinking GRS with bounded curvature}

From Corollary \ref{cor type I sol split}, we obtain our main theorem, which has some flavor of the
canonical neighborhood property.

\begin{theorem} \label{thm main de BC}
Let $(\mathcal{M}^{n},g(t),f(t))$, $t\in(-\infty,1)$, be a complete
noncompact $\kappa$-noncollapsed \emph{(}below a fixed scale\emph{)} shrinking GRS in
canonical form with bounded curvature. Then, for any $\varepsilon>0$ there
is a compact set $K_{\varepsilon}\subset\mathcal{M}$ such that for each $
x\in \mathcal{M}-K_{\varepsilon}$ there is an open set $\mathcal{U}\subset
\mathcal{M}$ such that $(\mathcal{U},g(t),x)$ is $\varepsilon$-close for $
t\in(-\varepsilon^{-1},1-\varepsilon)$ in the $C^{\lfloor\varepsilon^{-1}
\rfloor}$-pointed Cheeger--Gromov sense to $(B_{\varepsilon^{-1}}^{h(0)}(y)\times\left(
-\varepsilon ^{-1},\varepsilon^{-1}\right) ,h(t)+ds^{2},(y,0))$, where $y \in
\mathcal{N}^{n-1}$ and where $(\mathcal{N}, h(t))$, $t\in (-\infty,1)$, is a
complete \emph{(}possibly flat\emph{)} $\kappa$-noncollapsed \emph{(}below
some fixed scale\emph{)} Type I ancient solution of Ricci flow
 with $|\func{Rm}_{h(t)}| \leq \frac{C}{1-t}$.
\end{theorem}

\begin{proof}
Suppose that the theorem is false. Then there exists $\varepsilon >0$ and a
sequence of points $x_i\rightarrow\infty$ without the stated property. On
the other hand,  by  Corollary \ref{cor type I sol split} $\{(\mathcal{M},g(t),x_{i})\}$, $t\in(-\infty,1)$,
subconverges to the product of $\mathbb{R}$ with a complete
 $\kappa$-noncollapsed (below some fixed scale) Type I ancient
solution $(\mathcal{N}^{n-1},h(t))$. Choosing $i$ large enough leads to a
contradiction. ~~~ $\square$
\end{proof}

\begin{remark}
As explained in the introduction, in general the best result in the direction of
Theorem \ref{thm main de BC} one can hope for is that if curvature tensor
$\func{Rm}$ of  a complete noncompact shrinking GRS does not limit to
$0$ at infinity, then $(\mathcal{M}^n,g)$
splits as the product of $\mathbb{R}$
and a Riemannian manifold $(\mathcal{N}^{n-1},h)$.
\end{remark}

Now we give some conclusions related Theorem \ref{thm main de BC}
assuming that the solution is $\kappa$-noncollapsed below all scales.
First, in dimension $n=4$,  the sectional curvature of $h(t)$ is nonnegative
by Chen \cite{ChenStrongUniqueness}. If $\mathcal{N}^3$ is compact, then by Ni
\cite{NiClosedTypeI}, $(\mathcal{N},h(t))$ is a spherical space form (note that a
compact quotient of $\mathbb{R}^{3}$ and $\mathcal{S}^{2}\times\mathbb{R}$ are
not $\kappa$-noncollapsed). If $\mathcal{N}^3$ is noncompact, then $h(t)$ is
either $\mathbb{R}^{3}$, $\mathcal{S}^{2}\times\mathbb{R}$, its $\mathbb{Z}
_{2}$-quotient, or has positive sectional curvature.

For the general case $n\geq4$, we have outside a compact set $K$
that $\nabla f\neq0$ and $\nu=\frac{\nabla f}{|\nabla f|}$ is well defined.
Let $\Sigma _{c}=\{f=c\}$, which is a $C^{\infty}$ hypersurface, and let $
g^{T}=g-\nu^{\ast }\otimes\nu^{\ast}=g|_{\Sigma_{c}}$.
Let $F^{(i)}(x) \doteqdot 2( \sqrt{f}(x)-\sqrt{f}(x_i))$.
By the Gauss equations,
\begin{align*}
f^{1/2}\nabla^{2}F^{(i)} & = \nabla^{2}f-\frac{1}{2}\frac{\nabla
f\otimes\nabla f}{f} \\
& = -\limfunc{Rc}{}_{g^{T}}+\frac{1}{2}g^{T}-\func{Rc}_{g}(\nu,\nu)\nu^{
\ast}\otimes\nu^{\ast}-\func{Rm}_{g}(\nu,\cdot,\cdot,\nu) \\
& \quad\; +H\func{II}-\func{II}^{2}+\frac{R}{2f}\nu^{\ast}\otimes\nu^{\ast},
\end{align*}
where $\func{II}$ and $H$ are the second fundamental form and mean curvature
of $\Sigma_{c}$. By the theorem above, on $\mathcal{N}$ the limit $\lim_{i\rightarrow
\infty}(f^{1/2}\nabla^{2}F^{(i)})=\frac{1}{2}h -\func{Rc}_{h}$ exists using
Cheeger--Gromov convergence. Let $p\in\mathcal{M}$. We have $\frac{1}{2}
-\nabla^{2}f(\nu,\nu)=\frac {1}{2\left\vert \nabla f\right\vert }\nabla
R\cdot\nu=o(d(\cdot,p)^{-1})$ and $\nabla^{2}f(X,\nu)=-\frac{1}{2\left\vert
\nabla f\right\vert }\nabla R\cdot X=O(d(\cdot,p)^{-1})$ for unit $X\in
T\Sigma_{c}$.

\vskip .1cm
\textbf{Example}.  Let $(\mathcal{N}^{n-1},\bar{g},\bar{f})$ be a shrinking
GRS. Define the shrinking GRS $(\mathcal{M}^{n},g,f)$ by $\mathcal{M}=
\mathcal{N}\times\mathbb{R}$, $g=\bar{g}+ds\otimes ds$, and $f(y,s)=\bar {f}
(y)+\frac{s^{2}}{4}$. Then $f^{1/2}\nabla^{2}F^{(i)}=\frac {\bar{g}}{2}-
\func{Rc}_{\bar{g}}+\frac{R}{2f}\nu^{\ast}\otimes \nu^{\ast}=\bar{\nabla}^{2}
\bar{f}+O(d(\cdot,p)^{-1})$.

\vskip .1cm
There are implications of the works of Naber \cite{Naber} and Cao and Zhang
\cite{CaoX-ZhangQ} on Type I ancient solutions. In particular, Cao and Zhang
proved that, for any $\kappa$-noncollapsed complete Type I ancient solution
with nonnegative curvature operator, Naber's complete shrinking GRS backward
limit is necessarily nonflat. In view of Lemma 8.27 in \cite{ChowLuNi}, we
may apply this to those shrinking GRS which are singularity models with
nonnegative curvature operator.

\section{Standard point picking}
The results of this section are known consequences of standard techniques,
going back to Hamilton. For convenience, we include the statements in the form
we use.
Let $(\mathcal{M}^n,g)$ be a complete noncompact Riemannian manifold satisfying
\begin{equation}
 \func{ACR}\doteqdot\lim_{\rho \rightarrow\infty}\sup_{x\in
B_{\rho}(O)} d (O,x)^{2}|\func{Rm}(x)|=\infty \quad \text{for some } O
 \in \mathcal{M}. \label{eq ACR infinity}
\end{equation}
We now describe the point picking argument, which enables us to
construct sequences $x_{i}\rightarrow\infty$ such that $g_{i} \doteqdot |\func{Rm}_{g}(x_{i})|
 \cdot g$ satisfies (\ref{localcurv}) in the admissible assumption.
Let $\rho _{i}\rightarrow\infty$ be a sequence. Choose $x_{i}$ with $r_{i}\doteqdot
d_g(O,x_{i})$ so that
\begin{equation}
\frac{r_{i}\left( \rho_{i}-r_{i}\right) \sqrt{|\func{Rm}_g(x_{i})|}}{
\sup_{x\in B^g_{\rho_{i}}(O)}d_g(O, x)\left( \rho_{i}-d_g(O,x)\right) \sqrt {|\func{Rm}_g
(x)|}}\doteqdot1-\delta_{i}\rightarrow1. \label{eq x i choice}
\end{equation}

\textbf{Claim}. Assume (\ref{eq ACR infinity}). Then
\begin{equation*}
\alpha_{i}\doteqdot r_{i}\sqrt{|\func{Rm}_g(x_{i})|}\rightarrow \infty\quad\;
\text{and}\quad\;\omega_{i}\doteqdot(\rho_{i}-r_{i})\sqrt{|\func{Rm}_g(x_{i})|}
\rightarrow\infty.
\end{equation*}
\emph{Proof of the claim}: We compute that
\begin{align*}
\frac{1}{\alpha_{i}^{-1}+\omega_{i}^{-1}} & =\frac{\alpha_{i}\omega_{i}}{
\alpha_{i}+\omega_{i}} \\
& =\frac{r_{i}(\rho_{i}-r_{i})\sqrt{|\func{Rm}_g(x_{i})|}}{\rho_{i}} \\
& =\left( 1-\delta_{i}\right) \cdot \frac{\sup_{x\in B^g_{\rho_{i}}(O)}d_g(O, x)\left(
\rho_{i}- d_g(O, x)\right) \sqrt{|\func{Rm}_g(x)|}}{\rho_{i}} \\
& \geq\frac{1-\delta_{i}}{2} \cdot \sup_{x\in B^g_{\rho_{i}/2}(O)}d_g(O,x)\sqrt {|\func{Rm
}_g(x)|} \\
& \geq\frac{1}{3}\sup_{x\in B^g_{\rho_{i}/2}(O)}d_g(O, x)\sqrt{|\func{Rm}_g(x)|} \\
& \rightarrow\infty \quad \text{as } i \rightarrow\infty.
\end{align*}
 Hence $\alpha_{i} \rightarrow \infty$ and $\omega
_{i} \rightarrow \infty$.
~~~ $\square$\smallskip

Then, for $x \in B_{\rho_{i}}^{g}(O)=B_{\rho_{i}\sqrt{|\func{Rm}_g(x_{i})|}}^{g_{i}}(O)$,
\begin{align*}
& \sqrt{|\func{Rm}_{g_{i}}(x)|} \leq\frac{1}{1-\delta_{i}} \cdot \frac{r_{i}\left(
\rho_{i}-r_{i}\right) }{d_g(O, x)\left( \rho_{i}-d_g(O, x)\right) } \\
=& \frac{1}{1-\delta_{i}} \cdot \frac{\alpha_{i}\omega_{i}}{d_g(O,x)\sqrt {|\func{Rm}_g
(x_{i})|}\left( \rho_{i}-d_g(O, x)\right) \sqrt {|\func{Rm}_g(x_{i})|}} \\
 =& \frac{1}{1-\delta_{i}} \cdot \frac{\alpha_{i}}{\alpha_{i}+\left(
d_g(O,x)-r_{i}\right) \sqrt{|\func{Rm}_g(x_{i})|}} \cdot \frac{\omega_{i}}{\omega
_{i}-\left( d_g(O,x)-r_{i}\right) \sqrt{|\func{Rm}_g(x_{i})|}}.
\end{align*}

Let $\rho\in\left( 0,\infty\right) $. Suppose that $x\in
B_{\rho}^{g_{i}}(x_{i})$, which is equivalent to $x\in B_{\rho/\sqrt{|\func{Rm
}_g(x_{i})|}}^{g}(x_{i})$. This implies that
\begin{equation*}
-\rho/\sqrt{|\func{Rm}_g(x_{i})|}\leq d_g\left( O, x\right) -r_{i}\leq \rho/\sqrt{|
\func{Rm}_g (x_{i})|},
\end{equation*}
i.e.,
\begin{equation*}
-\rho\leq\left( d_g(O, x)-r_{i}\right) \cdot \sqrt{|\func{Rm}_g(x_{i})|}\leq \rho.
\end{equation*}
Thus:

\begin{lemma}. \label{lem only one of}
Assume \emph{(\ref{eq ACR infinity})} and choose $x_i$ as in \emph{(\ref{eq x i choice})}.
If $x\in B_{\rho }^{g_{i}}(x_{i})$, then
\begin{equation*}
\sqrt{|\func{Rm}_{g_{i}}(x)|}\leq \frac{1}{1-\delta _{i}} \cdot \frac{\alpha _{i}}{
\alpha _{i}-\rho } \cdot \frac{\omega _{i}}{\omega _{i}-\rho },
\end{equation*}
where the \textsc{rhs} tends to $1$ as $i\rightarrow \infty $.
\end{lemma}

\begin{corollary}
Suppose that $(\mathcal{M}^{n},g,f,1)$ is a complete noncompact $\kappa$-noncollapsed
 \emph{(}on all scales\emph{)}
shrinking GRS satisfying \emph{(\ref{eq ACR infinity})}
and $\left\vert \func{Rm}\right\vert (x) =o(d(O,x)^2)$, where $O$ is a minimum point of $f$.
Let $x_{i}\rightarrow \infty $ be chosen as in \emph{(\ref{eq x i choice})}.
Assume that $K_i \doteqdot \left\vert \func{Rm}\right\vert (x_{i})\geq c>0$
  and
  \begin{equation}
  \sup_i \left ( K_i^{-3/2} \cdot \sup_{x \in B_\rho(x_i)} |\nabla \func{Rc}| \right ) < \infty
  \quad \text{for all } \rho > 0. \label{eq nabla Rc assump}
  \end{equation}
  Define $g_i \doteqdot K_i g$. Let $\alpha$ be any constant in $(0,1)$.

 Then $(\mathcal{M}^n,g_{i},x_{i})$ converges in the $C^{2,\alpha }$ pointed
Cheeger--Gromov sense to a complete $C^{2,\alpha }$ Riemannian manifold $(\mathcal{M}
_{\infty }^{n},g_{\infty },x_{\infty })$, and $(\mathcal{M}_{\infty
},g_{\infty })$ is isometric to the product of $\mathbb{R}$ with a complete
nonflat Riemannian manifold $(\mathcal{N}^{n-1},h)$.
\end{corollary}

\begin{proof}
The corollary follows from the claim  that $\{ (\mathcal{M},g_{i}, f, K_i^{-1}, x_{i})\}$ is an admissible
sequence of complete noncompact shrinking GRS and from Theorem \ref{thm compactness shrinker seq}.

 To see the claim, note that condition 5 follows from Lemma \ref{lem only one of}.
Condition 4 follows from condition 5 and the   $\kappa$-noncollapsing assumption.
Condition 6 follows from assumption (\ref{eq nabla Rc assump}). Conditions 1 and 2 follow from
the definition and $\lambda_i = K_i^{-1} \leq c^{-1}$.

Finally, to see condition 3, we compute that
\[
\lambda_i d_{g_i}(x_i,O) = K_i^{-1} \cdot K_i^{1/2}  d_{g}(x_i,O)
=\frac{ d_{g}(x_i,O)}{K_i^{1/2}} = \frac{ d_{g}(x_i,O)}{o(d_g(x_i,O)} \rightarrow \infty,
\]
where we have used  $\left\vert \func{Rm}\right\vert (x) =o(d(O,x)^2)$ in the last equality.
~~~   $\square$
\end{proof}



\end{document}